\documentclass[12pt]{amsart}

\usepackage{amsmath, amssymb, amscd, eucal}
\usepackage{latexsym}
\usepackage{graphicx}

\usepackage{amsmath,amssymb,amsfonts,amsthm,amscd,indentfirst}

\usepackage{indentfirst}
\usepackage{hyperref}
\usepackage{graphicx}

\usepackage{color}
\usepackage[dvipsnames]{xcolor}

\usepackage{setspace}
\usepackage{amsmath, amssymb,amsthm, amsfonts, color}

\numberwithin{equation}{section}

\newtheorem{thm}{Theorem}[section] 

\newtheorem{prop}{Proposition}[section]
\newtheorem{cor}{Corollary}[section] 

\theoremstyle{definition}

\theoremstyle{remark}
\newtheorem{rem}{Remark}[section]

\newcommand{\Ric}{\mbox{Ric}}
\newcommand{\R}{\mathbb R}

\newcommand{\be}{\begin{equation}}
\newcommand{\ee}{\end{equation}}
\newcommand{\bee}{\begin{equation*}}
\newcommand{\eee}{\end{equation*}}
\newcommand{\bal}{\begin{aligned}}
\newcommand{\eal}{\end{aligned}}

\def\vh{\vspace{.1cm}}

\def\p{\partial}

\def\la{\langle}
\def\ra{\rangle}

\def\Pi{\displaystyle{\mathbb{II}}}

\def\m{\mathfrak{m}}

\def\c{\mathfrak{c}}

\def\gext{g_{_{ext}}}

\begin{document}

\title{Horizons in noncompact fill-ins of nonnegative scalar curvature}

\author{Pengzi Miao}
\address[Pengzi Miao]{Department of Mathematics, University of Miami, Coral Gables, FL 33146, USA}
\email{pengzim@math.miami.edu}

\author{Sehong Park}
\address[Sehong Park]{Department of Mathematics, University of Miami, Coral Gables, FL 33146, USA}
\email{sxp3226@miami.edu}


\begin{abstract}
Given a complete Riemannian metric of nonnegative scalar curvature on 
$\Sigma \times (-\infty, 0 ] $, where $\Sigma$ denotes a $2$-sphere, 
we exhibit conditions that imply the existence of a closed minimal surface homologous to the boundary. 
\end{abstract}

\maketitle

\markboth{Pengzi Miao and Sehong Park}{Horizons in Noncompact NNSC Fill-ins}

\section{Introduction} 

Let  $\Sigma  $ denote a smooth, closed, orientable surface  that is topologically a $2$-sphere. 
Let $ N = \Sigma \times ( - \infty, 0 ] $.
Given a complete Riemannian metric $g$ with nonnegative scalar curvature on $N$, we consider 
the existence of closed minimal surfaces in $(N,g)$ homologous to the boundary $ \p N$. 

The perspective we take on this problem relates to the study of quasi-local mass in general relativity. 
Given a Riemannian metric $\sigma$  and  a function  $H$ on $ \Sigma$, 
in the derivation of 
positivity properties  of the quasi-local mass associated to $(\Sigma, \sigma, H )$, 
one often assumes there exists a compact, nonnegative scalar curvature (NNSC) 
fill-in of $(\Sigma, \sigma, H )$ (see \cite{Bartnik89, ST02, LY03, WY09}  for instance).  
Here a connected Riemannian manifold with boundary is called a NNSC fill-in of 
 $(\Sigma, \sigma, H)$ if the manifold  has nonnegative scalar curvature and 
  the  boundary can be diffeomorphically identified with $\Sigma$ 
 in such a way that the induced metric agrees with $ \sigma$ and the mean curvature equals $H$. 
 
In \cite{LLU22}, Lee, Lesourd and Unger initiated a study of noncompact, NNSC fill-ins of $(\Sigma, \sigma, H)$. 
Among other things, they proved that the result of Shi and Tam \cite{ST02} on 
the positivity of the Brown-York mass 
remains valid if a compact NNSC fill-in is replaced by a complete, noncompact, NNSC fill-in. 

Suppose $ (U,h )$ is a noncompact, NNSC fill-in  of $ (\Sigma, \sigma, H)$ where $ H > 0 $.
If $(U, h )$ contains a horizon, which is defined as a closed minimal surface homologous to $\p U$, 
then, passing to the outermost horizon  relative to $\p U$, 
one may focus on the finite region enclosed by $ \p U$ and the outermost horizon.
Effect of such a horizon on the boundary  $ \p U$ was studied  in 
\cite{M07, Lu-M17}, which were inspired by the Riemannian Penrose inequality \cite{Bray02, HI02}.

A naturally related question is: 
given a noncompact  fill-in  $(U, h)$, how to detect  if there exists  a horizon in $(U, h)$?
In this paper, we consider this question in a setting 
where a fill-in $U$ has simple topology of $ N = \Sigma \times ( - \infty, 0 ] $. 

In the context of the Riemannian Penrose inequality, 
without reference to closed minimal surfaces,
Zhu \cite{Zhu23} recently has derived a mass-systole inequality
based on
the $\mu$-bubble theory of Gromov \cite{Gromov19}.
Prompted by Zhu's work, 
we will apply the result in \cite{Zhu23} 
to study the existence of horizons in a fill-in $(N, g)$. 

We fix some notations and terminology first. 
Throughout, let 
$$ \mathcal{S} = \left\{ S \subset N  \ | \ S \ \text{is a closed, embedded surface homologous to } \p N \right\}. $$ 
Given a metric $g$ on $N$, 
 $ S \in \mathcal{S}$ is called a horizon in $(N,g)$  if 
$ S$ is a minimal surface. 
Let  $ t $ denote the standard coordinate on $(- \infty, 0]$.
Let 
$$ A_{_{-\infty} } (g)   = \inf \left\{ \liminf_{k \to \infty} \, | S_k | \ | \ \{ S_k \} \subset \mathcal{S},  \ 
\lim_{ k \to \infty} \max_{S_k} t = - \infty     \right\}. $$
Here $ | S |$ denote the area of a surface $ S$ in $(N,g)$. 
Equivalently, 
$$ A_{_{-\infty}} (g)   = \lim_{s \to - \infty} A_s (g) ,  $$
where $ A_s  (g) =  \inf \left\{ \, | S | \ | \  S \subset \mathcal{S}, \  \max_{S} t \le s   \right\}  $.
We also let $ A(g) = A_0(g)$, the infimum of  $|S| $ for all $ S \in \mathcal{S}$.

The boundary $ \p N$ of $(N,g)$ is called mean-convex if 
the mean curvature vector of $ \p N$  points to the inside of $N$. 
In terms of our sign convention for the mean curvature, this means 
$\p N$ has positive mean curvature with respect to the outward normal.

\begin{thm} \label{thm-main-horizon-via-sigma}
Suppose $(N,g)$ is complete with nonnegative scalar curvature,
with mean-convex boundary.
Let $ \sigma $ be the induced metric  on $ \Sigma = \p N$. 
\begin{enumerate}

\vh

\item[(i)] If  $ \sigma  $ is  a round metric, i.e.  a metric of constant Gauss curvature, then
$(N,g)$ has a horizon if 
\be \label{eq-round-sigma}
 A_{ - \infty} (g)^\frac12  > 
| \Sigma |^\frac12  \, \left[  1 -  \frac{   \left( \int_{\Sigma }  H   \, d \sigma  \right)^2 }{16 \pi | \Sigma  |  }  \right] .
\ee
Here $ H $ is the mean curvature of $ \Sigma$ in $(N,g)$. 

\vh 

\vh

\item[(ii)] If  $ \sigma $ has positive Gauss curvature, then
$(N,g)$ has a horizon if 
\be \label{eq-notround-sigma}
 A_{ - \infty} (g)^\frac12  >  \frac{1}{ \sqrt{  4 \pi  }  } 
 \int_{\Sigma }  (H_0 - H) \, d \sigma .
\ee
Here $ H_0$ is the mean curvature of the isometric embedding of $ (\Sigma, \sigma)$ in $ \R^3$.

\vh 

\vh

\item[(iii)] For an arbitrary $\sigma$, $(N,g)$ has a horizon if 
$$  A_{ - \infty} (g)  \ge  | \Sigma  | . $$

\end{enumerate}

\end{thm}

\begin{rem}
Condition 
\eqref{eq-round-sigma}  is nearly sharp 
as one may construct  rotationally symmetric examples, by perturbing a standard cylinder for instance, 
which are free of horizons and have $A_{-\infty} (g)^\frac12 $ less than but arbitrarily close to the quantity 
on the right side of \eqref{eq-round-sigma}.
\end{rem}

\begin{rem}
We do not expect \eqref{eq-notround-sigma} to be sharp, 
as  indicated by its proof later.
\end{rem}

\begin{rem}
It is conceivable that case (iii) may also follow from the theory of mean curvature flow,
together with  \cite[Theorem 8.11]{GL83}.
\end{rem}

We will derive Theorem \ref{thm-main-horizon-via-sigma} in the following way: 
given an $(N, g)$, derive a lower estimate of the Bartnik quasi-local mass of its boundary 
data $(\p N, \sigma, H)$ via $A(g)$; 
if  $(N,g)$ has no horizons, show that $A(g)$ is equal to $A_{-\infty} (g)$.
If these are completed, the result then follows from existing known upper estimates of the Bartnik mass. 

We recall the  concept of the Bartnik  mass \cite{Bartnik89} below. 
 Given a  triple $(\Sigma, \sigma, H)$, the associated Bartnik mass 
can be defined as 
$$ \m_{_B} (\Sigma, \sigma, H) = \inf \left\{ \m (\gext ) \ | \ (E, \gext) \ \text{is an admissible extension of}  \ (\Sigma, \sigma, H) \right\}. $$ 
Here  $\m(  \, \cdot \, ) $  denotes the ADM mass \cite{ADM61}, and an asymptotically flat Riemannian $3$-manifold
$(E, \gext)$ with boundary $\p E$ is an admissible extension of  $ (\Sigma, \sigma, H)$  if

\begin{enumerate}
\item[i)]  $E$ is topologically $ \R^3 $ minus a ball and $\gext$  has nonnegative scalar curvature;
\item[ii)]  $ \p E$ with the induced metric is isometric to $ (\Sigma, \sigma)$,  and, under the isometry,
the mean curvature of $ \p E$  in  $ (E, \gext)$  equals $H$; and

\item[iii)] $ \p E$ is outer-minimizing in  $(E, \gext)$.

 \end{enumerate} 
We refer readers to \cite{AJ19, J19, McCormick20} for a number of variants of $\m_{_B} ( \, \cdot \, )$.

The following lower estimate on $\m_{_B} (\cdot)$ can be derived 
applying the work of   Lee-Lesourd-Unger \cite{LLU22} and Zhu \cite{Zhu23}. 

\begin{prop} \label{prop-bmass-lower-est}
Suppose $(N,g)$ is complete with nonnegative scalar curvature,
with mean-convex boundary.  
Let $ \sigma $ be the induced metric  on $ \Sigma = \p N$ 
and $H$ be  the mean curvature of $ \Sigma $ in $ (N,g)$. 
Then 
$$   \m_{_B} (\Sigma, \sigma, H )  \ge  \sqrt{ \frac{ A (g) }{ 16 \pi} }  . $$  
\end{prop}


If  $(N,g)$ has no horizons, one can show 

\begin{prop} \label{prop-nohorizon-NG}
Suppose $(N,g)$ is complete with nonnegative scalar curvature,
with mean-convex boundary.
If $(N,g)$ has no horizons, 
then
$$    A_{ - \infty} (g)  = A (g) . $$  
\end{prop}

Combined together, Propositions \ref{prop-bmass-lower-est} and \ref{prop-nohorizon-NG} directly  imply  

\begin{thm} \label{thm-main-aleft-area}
Suppose $(N,g)$ is complete with nonnegative scalar curvature, with mean-convex boundary. 
Let $ \sigma $ be the induced metric  on $ \Sigma = \p N$ 
and $H$ be  the mean curvature of $ \Sigma $ in $ (N,g)$. 
If
$$    \sqrt{ \frac{ A_{_{-\infty} }  (g) }{16 \pi} }  > \m_{_{B} } (\Sigma, \sigma, H)  ,    $$
then $(N,g)$ contains a horizon.
 \end{thm}

If $ \sigma $ is a metric of positive Gauss curvature and $H $ is a positive function, 
it follows from the work of Shi-Tam \cite{ST02} that
\be \label{eq-mb-mby}
 \m_{_B} (\Sigma, \sigma, H ) \le \frac{1}{8\pi} \int_\Sigma (H_0 - H ) \, d \sigma . 
\ee
If  $ \sigma = \sigma_o$ is a round metric, it was further known \cite{M07, MX19}   that
\be \label{eq-mb-sigma-round}
 \m_{_B} (\Sigma, \sigma_o, H ) \le \sqrt{  \frac{ |\Sigma |  } { 16 \pi}  } 
 \left[ 1 - \frac{1}{ 16 \pi | \Sigma |   } \left( \int_\Sigma H \, d \sigma_o \right)^2 \right]. 
\ee
Cases (i) and (ii) of Theorem \ref{thm-main-horizon-via-sigma} follow from Theorem \ref{thm-main-aleft-area}
and the existing estimates \eqref{eq-mb-sigma-round} and \eqref{eq-mb-mby}, respectively. 

Case (iii) of Theorem \ref{thm-main-horizon-via-sigma} 
follows from Proposition \ref{prop-nohorizon-NG} and 
the fact $ A(g) < | \Sigma | $. 

We want to mention that in the case of a compact fill-in of $(\Sigma, \sigma, H),$ results of a similar nature to Theorem \ref{thm-main-aleft-area} were obtained by Shi-Tam \cite{ST07}.

In the applications below, we consider the case when 
$(N,g)$ has a positive capacity.  

\begin{thm} \label{thm-Ag-by-Cg}
Suppose $(N,g)$ is complete with nonnegative scalar curvature, with mean-convex boundary. 
Suppose there exists a harmonic function $ w $ on $(N,g)$ with 
$ w = 1  $ at $  \p N $  and $ 0 < w  < 1 $  in the interior of $N$. 
If 
\be \label{eq-Arealim-gdlim}
 \limsup_{s \to -\infty} | \nabla w | < \frac14 \c \, \m^{-2}_{_{B} } (\p N)  ,    
\ee
then $(N,g)$ contains a horizon.
Here $ \c > 0 $ is the constant given by
$$ \int_{\p N} | \nabla w | \, d \sigma = 4 \pi \c , $$
and 
$ \m_{_B} (\p N)$ denotes the Bartnik mass of the boundary data of $(N, g)$.
\end{thm}

\begin{rem}
The capacity of  a surface $S \in \mathcal{S}$,  relative to $ \p N$,    can be defined as 
$$ \c_{_S}  : = \frac{1}{4 \pi}  \int_{D_{_S}}  | \nabla u_{_S} |^2 \, d V  
= \frac{1}{4\pi}  \int_{\p N} | \nabla u_{_S} | d \sigma, $$
where $ u_{_S}$ is the unique $g$-harmonic function on $D_{_S}$, the region bounded by $ S$ and $ \p N$, 
satisfying 
$ u_{_S} = 1 $ at $ \p N$ and $ u_{_S} = 0 $ at $ S$.
The capacity of $(N,g)$, relative  to $\p N$,  can then be defined as 
$$ \c (g) = \inf  \left\{ \c_{_S}  \ | \ S \in \mathcal{S} \right\} . $$ 
It can be checked that 
  $ \c (g) > 0 $  
if and only if there exists a harmonic function $ w $  on $(N, g)$ such that 
$ w  = 1  $ at $  \p N $ and  $ 0 < w < 1 $  in the interior of $N$.
\end{rem}

Without reference to $\m_{_B} (\cdot) $, the following holds.

\begin{thm} \label{thm-Ag-by-Cg-2}
Suppose $(N,g)$ is complete with nonnegative scalar curvature, with mean-convex boundary. 
Suppose there exists a harmonic function $ w $ on $(N,g)$ with 
$ w = 1  $ at $  \p N $  and $ 0 < w  < 1 $  in the interior of $N$. 
If 
\be \label{eq-Arealim-gdlim-2}
 \limsup_{s \to -\infty} | \nabla w | \le  \frac{1}{ | \p N | } \int_{\p N} | \nabla w | \, d \sigma , 
\ee
then $(N,g)$ contains a horizon.
\end{thm}

Together with Cheng-Yau's gradient estimate \cite{CY75}, the above implies 

\begin{cor} \label{cor-pc-afric}
Suppose $(N,g)$ is complete with nonnegative scalar curvature, with mean-convex boundary. 
Suppose $ \c (g) > 0 $ and
\be \label{eq-Arealim-gdlim-3}
 \lim_{s \to - \infty}    \min  \Ric_- (x)    = 0  . 
\ee
Here $ \min \Ric_- (x)  $ denotes the minimum negative Ricci curvature of $g$ at  $ x \in N$.
Then $(N,g)$ contains a horizon.
\end{cor}

A basic example here is the spatial Schwarzschild fill-in of a round sphere with constant mean curvature.
On $ \R^3 \setminus \{ 0 \} $, consider the spatial Schwarzschild metric 
$$ g_m = \left( 1 + \frac{ m }{2 | x | } \right)^4 g_0 $$
 where $ m > 0 $ is a constant and $g_0$ is the Euclidean metric. 
Given any $ r > \frac{m}{2}$, let 
$$ N = \{ x \in \R^3 \ | \ 0 < | x | \le r \} . $$
$(N, g_m)$ satisfies conditions in Corollary \ref{cor-pc-afric}, 
and has a horizon at $ | x |  = \frac{m}{2} $.

\section{Proof of Propositions \ref{prop-bmass-lower-est}, \ref{prop-nohorizon-NG} 
and Theorems \ref{thm-Ag-by-Cg}, \ref{thm-Ag-by-Cg-2}} 

Before giving the proofs, 
we recall the concept of asymptotically flat manifolds and their mass. 
A smooth Riemannian $3$-manifold $(M, g)$ is called asymptotically flat if $M$,  
outside a compact set, is diffeomorphic to $ \R^3$ minus a ball;
the associated metric coefficients  satisfy   
$$
g_{ij} = \delta_{ij} + O ( |x |^{-\tau} ), \ 
 \p g_{ij} = O ( |x|^{-\tau -1}) ,  \ \ \p \p g_{ij} = O (|x|^{-\tau -2} ), 
$$
for some $ \tau > \frac12$; and the scalar curvature  of $g$ is  integrable.
Under these AF conditions, the limit 
$$
\m (g) : = \lim_{ r \to \infty  } \frac{1}{16 \pi} \int_{ |x | = r } \sum_{j, k} ( g_{jk,j} -  g_{jj,k} ) \frac{x^k}{ |x| } 
$$
exists and is called the ADM mass  \cite{ADM61} of $(M,g)$. 
It is a result of Bartnik \cite{Bartnik86} and of Chru\'{s}ciel \cite{Chrusciel86}
that $\m (g) $ is independent on the choice of the coordinates $\{ x_i \}$.
The Riemannian positive mass theorem, 
first proved by Schoen and Yau \cite{SchoenYau79}, and by Witten \cite{Witten81},
asserts that, if  $(M, g)$ is complete, asymptotically flat $3$-manifold with nonnegative scalar curvature without boundary, 
then $ \m (g)  \ge 0  $
and equality  holds if and only if $(M, g)$ is isometric to the Euclidean space $ \R^3$. 

We first  prove Proposition \ref{prop-bmass-lower-est}. 
The idea is simple. 
Given an admissible extension $(E,\gext)$, one may glue $(N,g)$ and $(E,\gext)$ along $\Sigma$ 
to obtain a manifold $(M,\tilde g)$.  Using a result of Lee-Lesourd-Unger \cite{LLU22}, 
one can approximate $\tilde g$ by complete metrics $\tilde g_\epsilon$ of  nonnegative scalar curvature without changing much the mass. 
Applying Zhu’s mass-systole inequality \cite{Zhu23} to $\tilde g_\epsilon$ then produces a lower bound of 
the mass in terms of $A(\tilde g)$, 
which can be shown to equal $A(g)$ via the outer-minimizing property of $\Sigma$.

\begin{proof}[Proof of Proposition \ref{prop-bmass-lower-est}]
Given $(E, \gext)$, an admissible extension of $(\Sigma,\sigma,H)$, 
let $(M, \tilde g)$ be 
  the manifold that is 
obtained by gluing $(N, g)$ and $(E, \gext) $ along their
common boundary $ (\Sigma, \sigma) $. 
By \cite[Theorem 3.3]{LLU22}, for small $\epsilon>0$,  there exists a family of metrics $\{ \tilde g_{\epsilon} \}$
with nonnegative scalar curvature on $M$, satisfying 
\begin{itemize}
\item $\|\tilde g_\epsilon - \tilde g\|_{C^{0}} < \epsilon$;
\item $|\m(\tilde g_\epsilon) - \m(\tilde g)| < \epsilon$.
\end{itemize}
The first condition implies 
there is an absolute constant $C>0$ such that 
\begin{equation}\label{eq-area-compare}
 (1-C\epsilon)|S|_{\tilde g} \le |S|_{\tilde g_\epsilon} \le (1+C\epsilon)|S|_{\tilde g}   
\end{equation}
for any closed surface $S\subset M$. 
Taking the infimum of \eqref{eq-area-compare} over  $S\in\mathcal S_M $,
the set of all closed surfaces in $M$ that is homologous to $ \Sigma$,
we obtain 
\begin{equation}\label{eq-A-eps-compare}
(1-C\epsilon)A(\tilde g) \le A(\tilde g_\epsilon) \le (1+C\epsilon)A(\tilde g) , 
\end{equation}
where 
$$
A(\tilde g)
  = \inf\big\{|S|_{\tilde g} \ | \ S \in \mathcal{S}_M \big\}, \ 
  A(\tilde g_\epsilon)
  = \inf\big\{|S|_{\tilde g_\epsilon} \ | \ S \in \mathcal{S}_M \big\}.
$$
 
As  $M\cong \mathbb R^{3}\setminus\{0\}$, 
the mass-systole inequality in  \cite[Theorem 1.4] {Zhu23} applies to $\tilde g_\epsilon$, which gives 
\be \label{eq-mass-systole-1}
\m(\tilde g_\epsilon) \ge \sqrt{\frac{A(\tilde g_\epsilon)}{16\pi}}.
\ee
It follows from \eqref{eq-A-eps-compare}, \eqref{eq-mass-systole-1} and the fact  $|\m(\tilde g_\epsilon) - \m(\tilde g)| < \epsilon$ that 
$$
\m(\tilde g) 
\ge \m(\tilde g_\epsilon)-\epsilon 
\ge \sqrt{\frac{A(\tilde g_\epsilon)}{16\pi}} - \epsilon 
\ge \sqrt{\frac{(1-C\epsilon)A(\tilde g)}{16\pi}} - \epsilon.
$$
Letting $\epsilon\to 0$ yields
\begin{equation}\label{eq:mass-lower-A-tildeg}
    \m(\tilde g) \ge \sqrt{\frac{A(\tilde g)}{16\pi}}.
\end{equation}

Note that $\Sigma = \p M^+$ is outer-minimizing in $(E,\gext)$.
This implies, whenever a bounded region $\Omega $ in $M$  protrudes into $M^+$, 
 $ | \p \Omega | \ge | \p ( \Omega \cap N ) |$.
 From this, it follows  
 \be \label{eq-agn-agm}
 A(g) = A(\tilde g) . 
\ee
We conclude from  \eqref{eq:mass-lower-A-tildeg} and \eqref{eq-agn-agm} that 
$$ \m(\gext)=\m(\tilde g) \ge \sqrt{\frac{A(g)}{16\pi}} , $$
which proves Proposition\ref{prop-bmass-lower-est}.
\end{proof}

Next, we prove Proposition \ref{prop-nohorizon-NG}. 

\begin{proof}[Proof of Proposition \ref{prop-nohorizon-NG}]
Under the assumption that $(N,g)$ has no horizons, 
we will examine the behavior of the area outer-minimizing closure of a sequence of surfaces that drift to $ s = -\infty$. 
For each integer $n \ge 1$, let
$$ S_n = \Sigma \times \{ -n \} \ \  \text{and} \ \  \Omega_n = \Sigma \times [-n,0] \subset N .$$
Since $\partial N$ is mean-convex, there exists a surface $ \Sigma_n$ that is 
the area outer-minimizing closure of $S_n$
(also known as  the boundary of the outer-minimizing hull of $S_n$)  with respect to $ \p N$  in $\Omega_n$.
The surface $ \Sigma_n$ satisfies 
\be \label{eq-Sigman-omin}
 | \Sigma_n | \le | \tilde \Sigma | 
\ee
for every closed surface $\tilde{\Sigma} \subset \Omega_n$ homologous to  $\partial N$.
By  \cite[Theorem 1.3(iii)]{HI02}, $\Sigma_n$ is  $C^{1,1}$ and $ \Sigma_n \setminus S_n$ 
is smooth and area minimizing in the interior of $\Omega_n$. 

Since  $ \Sigma_n \subset \Omega_n \subset \Omega_{n+1}$, 
$\Sigma_n$ is a competitor for $ \Sigma_{n+1}$ in $\Omega_{n+1}$, which implies 
$$ |\Sigma_{n+1}| \le |\Sigma_n |. $$
Given any $
S\in \mathcal{S}$, there is an integer $N $ so that 
$S \subset \Omega_{N} $. 
Hence,  
$$|\Sigma_n| \le |S|, \ \forall \, n \ge  N . $$ 
As a result, 
$  \lim_{n\to \infty} |\Sigma_n|  \le |S | $.
Taking the infimum over  $S \in \mathcal{S}$, we have 
$$ \lim_{n\to \infty} |\Sigma_n| \le  A(g). $$
Together with the fact $ A(g) \le | \Sigma_n |$, this shows 
\be \label{eq-Sigma-n-area-limit}
\lim_{n\to \infty} |\Sigma_n| =   A(g).
\ee

To proceed, we note that, if there exists an $n$ such that $ \Sigma_n $ is disjoint from $ S_n $, then
$ \Sigma_n = \Sigma_n \setminus S_n $ is a closed minimal surface homologous to $\p N$. 
Therefore, under the assumption that $(N,g)$ has no horizon, 
$ \Sigma_n \cap S_n \ne \emptyset $, $ \forall \, n $.

\vspace{.2cm}

There are two remaining  cases about $\{ \Sigma_n \}$.

\medskip 

\noindent \textbf{Case 1}.  $\exists $ a subsequence $\{ n_k \}$ such that $\displaystyle \lim_{k  \to \infty} \max_{\Sigma_{n_k} }  t = -\infty$.
In this case,  
$$  A(g) = A_{-\infty}(g)  $$
 by  \eqref{eq-Sigma-n-area-limit} and the definition of $A_{-\infty} (g)$.

\bigskip

\noindent \textbf{Case 2.}   $\exists $ some $ T > 0$ such that  
$ \Sigma_{n} \cap \Omega_{_T}  \ne \emptyset $, $ \forall \,  n $.
Let $  p_n \in \Sigma_n \cap \Omega_{_T}$. 
Passing to a subsequence, we may assume $ p_n \to p_0$ for some $p_0 \in \Omega_{_T}$.

Let $ \Sigma^*_n $ denote the connected component of $ \Sigma_n \setminus S_n$ which contains $p_n$. 
We claim $ \Sigma_n^*$ can not be a closed surface. 
If $ \Sigma_n^*$ is a closed surface,   then either $ \Sigma_n^*$ is homologous to $ \p N$, or 
$ \Sigma^*_n = \p D_n$ for some finite domain $D_n \subset N$. Here we made use of the 
topology of $ N = S^2 \times (-\infty, 0]$. The first case would not occur since otherwise $\Sigma_n^*$ would be 
a competitor for $\Sigma_n$ and  $ |  \Sigma_n^* | < | \Sigma_n | $. 
The second case would not occur neither since otherwise $ \Sigma_n \setminus \Sigma_n^*$ would be a 
competitor for $ \Sigma_n$ and $ | \Sigma_n \setminus \Sigma_n^* | < | \Sigma_n | $.

Therefore, $ \Sigma_n^*$ is a noncompact area minimizing surface in the interior of $ \Omega_n$, whose topological
closure in $N$  intersects with  $S_n$.  
Since  $ \Sigma_n^*$ intersects with a fixed compact set $\Omega_{_T}$ for every $n$, 
similar to  Schoen-Yau's proof of the positive mass theorem \cite{SchoenYau79}, 
passing to a subsequence, $ \Sigma_n^*$ converges in $C^2$ norm 
to a noncompact area minimizing surface  $\Sigma^*$ on any compact sets of $N$. 
By the lower semicontinuity of area, 
$$ | \Sigma^* | \le \liminf_{n \to \infty} | \Sigma_n^* |  \le \lim_{n \to \infty} | \Sigma_n | = A(g) \le | \partial N | < \infty . $$
However, if the scalar curvature of $g$ is nonnegative, 
by a result  of Gromov-Lawson \cite[Theorem 8.11]{GL83}, $ | \Sigma^* | = \infty $.  
This gives a contradiction. 

We therefore conclude  $A_{-\infty}(g) = A(g)$.
\end{proof}

\begin{proof}[Proof of Theorems \ref{thm-Ag-by-Cg} and \ref{thm-Ag-by-Cg-2}]
By \eqref{eq-Arealim-gdlim}, 
$  \m_{_{B} } (\p N) < \infty  $  and  $  \limsup_{s \to -\infty} | \nabla w | < \infty $. 
The latter implies  $\sup_N |\nabla w| < \infty.$ Given any $ S \in \mathcal{S}$,  let $ D$ be the region bounded by $ S$ and $ \p N$.
Let $ \nu $ be the unit normal along $ S$ pointing to the interior of $D$.
Then
\be \label{eq-u-flux}
 \int_{S} \la \nabla w , \nu \ra d \sigma  = \int_{\p N} | \nabla w |  d \sigma  = 4 \pi \c  .
\ee
Given any $ s < 0 $, consider $ S \in \mathcal{S}$ satisfying  
$ \max_S t \le s $. 
By \eqref{eq-u-flux},
\bee
4 \pi \c \le   | S | \, \max_{S} | \nabla w | 
\le  |S | \, \sup_{\Sigma \times ( - \infty, s ] } | \nabla w |  .
\eee
Taking the infimum over $ S \in \mathcal{S}$ with $ \max_{S} t \le s $, one has 
\bee
4 \pi \c  \le  A_s (g) \, \sup_{\Sigma \times ( - \infty, s ] } | \nabla w |  .
\eee
Letting $ s \to - \infty$ gives 
\be \label{eq-est-Aninfty}
4 \pi \c  \le  A_{ - \infty}  (g) \, \limsup_{s \to - \infty}  | \nabla w |  .
\ee

To proceed, it suffices to assume $ A_{- \infty}(g) < \infty$. 
In this case,  by \eqref{eq-est-Aninfty} and \eqref{eq-Arealim-gdlim}, 
\be \label{eq-est-Aninfty-1}
16 \pi  \m^2_{_{B} } (\p N)   <  A_{ - \infty}  (g) .
\ee
Theorem \ref{thm-Ag-by-Cg}  follows from 
\eqref{eq-est-Aninfty-1} and 
Theorem \ref{thm-main-aleft-area}. 

Without reference to $\m_{_B} (\cdot) $, \eqref{eq-est-Aninfty} and \eqref{eq-Arealim-gdlim-2} imply
\be  \label{eq-est-Aninfty-2}
 A_{ - \infty}  (g) \ge | \p N |  . 
\ee
Theorem \ref{thm-Ag-by-Cg-2}  follows from 
\eqref{eq-est-Aninfty-2} and 
case (iii) of Theorem \ref{thm-main-horizon-via-sigma}. 
\end{proof}
 
\begin{proof}[Proof of Corollary \ref{cor-pc-afric}]
The assumption  of $ \c (g) > 0 $ implies that  there exists  a harmonic function $ w $ on $(N,g)$ with 
$ w = 1  $ at $  \p N $  and $ 0 < w  < 1 $  in the interior of $N$
(see \cite{BZ25, M24} for instance).
By Cheng-Yau's  gradient estimate \cite{CY75}, 
\eqref{eq-Arealim-gdlim-3} implies 
$$  \limsup_{s \to - \infty}  | \nabla w |   = 0 .  $$
It follows from Theorem   \ref{thm-Ag-by-Cg} or 
Theorem   \ref{thm-Ag-by-Cg-2} that $(N,g)$ has a horizon. 
\end{proof}

\section*{Acknowledgments}
The authors thank the anonymous referee for the valuable feedback, which greatly improved the results in this paper.

\end{document}